\documentclass{article}
\pdfoutput=1
\usepackage{import}
\usepackage{graphicx}
\usepackage{amsthm}
\usepackage{amsfonts}
\usepackage{enumitem}
\usepackage{amsmath}
\usepackage{amssymb}
\usepackage[hidelinks]{hyperref}
\usepackage{url}
\usepackage{pst-node}
\usepackage{tikz-cd}  	
\usepackage[all]{xy}
\usepackage{tikz}
\usepackage{xcolor}
\usepackage{stmaryrd}
\usepackage{enumitem}
\usepackage{rotating}
\newcommand*{\isoarrow}[1]{\arrow[#1,"\rotatebox{90}{\(\sim\)}"
]}
\usetikzlibrary{matrix}
\usepackage{etoolbox}
\theoremstyle{definition}
\newtheorem{theorem}{Theorem}[section]

\newtheorem{proposition}[theorem]{Proposition}
\newtheorem{lemma}[theorem]{Lemma}
\newtheorem{corollary}[theorem]{Corollary}

\newtheorem{formula}[theorem]{Formula}
\newtheorem{definition}[theorem]{Definition}
\newtheorem{example}[theorem]{Example}

\newtheorem{remark}[theorem]{Remark}
\newtheorem{convention}[theorem]{Convention}
\DeclareMathOperator*{\supp}{supp}
\DeclareMathOperator{\Hom}{{\text{Hom}}}
\DeclareMathOperator{\IC}{{\textbf{IC}}}
\DeclareMathOperator{\Ext}{{\text{Ext}}}
\DeclareMathOperator{\End}{{\text{End}}}
\DeclareMathOperator{\RHom}{{\text{RHom}}}
\DeclareMathOperator{\Diff}{{\text{Diff}}}
\DeclareMathOperator{\RSHom}{{\mathcal{RH}\text{om}}}
\DeclareMathOperator{\ESxt}{{\mathcal{E}\text{xt}}}
\DeclareMathOperator{\REnd}{{\text{REnd}}}
\DeclareMathOperator{\Ker}{{\text{Ker}}}

\DeclareMathOperator{\Spec}{{\text{Spec}}}
\DeclareMathOperator{\ddeg}{{\text{deg}}}
\DeclareMathOperator{\heart}{\ensuremath\heartsuit}
\DeclareMathOperator{\Id}{Id} 
\newcommand{\Modu}{{\text{-mod}}}

\title{\textbf{On the derived ring of differential operators on a singularity}}
\author{Haiping Yang\footnote{Department of Mathematics, Imperial College London, Exhibition Rd, South Kensington, London SW7 2BX, \href{mailto:hy2313@ic.ac.uk}{hy2313@ic.ac.uk}}}
\date{}
\begin{document}
\maketitle

\begin{abstract}
We show for an affine variety $X$, the derived category of quasi-coherent $D$-modules is equivalent to the category of DG modules over an explicit DG algebra, whose zeroth cohomology is the ring of Grothendieck differential operators $\Diff(X)$. When the variety is cuspidal, we show that this is just the usual ring $\Diff(X)$, and the equivalence is the abelian equivalence constructed by Ben-Zvi and Nevins. We compute the cohomology algebra and its natural modules in the hypersurface, curve and isolated quotient singularity cases. We identify cases where a $D$-module is realised as an ordinary module (in degree 0) over Diff($X$) and where it is not.
\end{abstract}

\section{Introduction}

Suppose $X$ is a smooth complex affine variety. Then $D$-modules on $X$ are defined to be modules over the ring of Grothendieck differential operators $\Diff(X)$ which behaves nicely. It is well known that $\Diff(X)$ is Noetherian in this case. If $X$ is not affine but still smooth, we can sheafify this construction to obtain a sheaf $\mathcal{D}_X$, and define the $D$-modules as the sheaves of modules over this sheaf of rings. 

For smooth varieties, this is enough. However, for singular varieties many problems can occur. The fundamental issue is that the ring of differential operators can be very complicated and sometimes not even Noetherian; a non-Noetherian example is the cubic cone $x_1^3 + x_2^3 + x_3^3 = 0$, see \cite{NotNeo}. Even when it is Noetherian, in general this construction does not have desirable geometric properties. For instance, Kashiwara's theorem, that for a closed embedding $Y\hookrightarrow X$, $D$-modules on $X$ set-theoretically supported on $Y$ are equivalent to $D$-modules on $Y$, fails in general. 

A typical solution is to \textit{define} the category so that this statement holds. But then, this is no longer the category of modules over any ring. This leads to two definitions of $D$-modules. Another one, called \textit{crystals} were first introduced by Grothendieck in \cite{gr} where he defined \emph{crystalline topos} and related to de Rham theory; Beilinson and Drinfeld in \cite{BD} related crystals to $D$-modules and noted that there is no distinction between smooth and non-smooth settings, and in \cite{D-modules_and_crystals} Gaitsgory and Rozenblyum gave a modern treatment. Crystals are defined as sheaves on $X$ which are equipped with compatible extensions (analogously to parallel transport) on infinitesimal thickenings. These three definitions coincide for smooth varieties. More generally, it is shown in \cite{DIFFERENTIAL_OPERATORS_ON_AN_AFFINE_CURVE} and generalised in \cite{Cusps_and_D-modules}, when the variety is Cohen--Macaulay and a cuspidal curve (or more generally there is a \textit{good cuspidal quotient morphism} from a smooth variety to it), these three definitions coincide.

One approach to singular varieties is via `\textit{derived algebraic geometry}' - roughly, this replaces ordinary rings by DG rings, and categories by DG (or triangulated) categories. In this context, it turns out that all three definitions, suitably interpreted, coincide, and give a DG category of $D$-modules, which we call $D\Modu_{X}^{dg}$.

In this work, we consider the relationship between this DG category and the original viewpoint of rings of Grothendieck differential operators. We prove that the derived category of $D$-modules is equivalent to the category of DG modules over an explicit DG algebra $\Diff(X)^{dg}$, which `corrects' Grothendieck's ring. Namely, it is concentrated in non-negative degrees, with zeroth cohomology equal to $\Diff(X)$, and with cohomology bounded by the dimension of the variety (so it is a nilpotent extension of the original ring).\\


The starting point of this work is the observation that there is a canonical $D$-module $D_X$ on (singular) varieties which is a compact generator of the relevant derived category of $D$-modules on $X$. The compact generation of $D_X$ was already proved in \cite{D-modules_and_crystals} in the framework of crystals, but we give an elementary proof in terms of modules over rings of differential operators.

We make use of this compact generator $D_X$, to show that: \begin{theorem}\label{introtheorem1}
There is an derived equivalence between the derived category of $D$-modules on $X$ and DG modules over $\Diff(X)^{dg}$, $$D\Modu_X^{dg}\cong \Diff(X)^{dg}\Modu.$$
\end{theorem}

In particular, in the case $X$ is smooth or more generally cuspidal, we show $\Diff(X)^{dg}\cong \Diff(X)$ is concentrated in degree zero. And we recover the fact that, in cuspidal/smooth cases, all definitions of $D$-modules coincide even in the underived setting. When $X$ need not be cuspidal, the idea to study $D$-modules as DG modules over a DG algebra is new, to our knowledge. \\

The structure of the paper is organised as follows: 

In Section \ref{$D$-modules as modules over a DG algebra}, we prove Theorem \ref{introtheorem1}. We write down an explicit compact generator in the affine case. We also explain how our construction agrees with \cite{Cusps_and_D-modules} in the cuspidal case.

To do this, we recall the canonical $D$-module $D_X$. We give an elementary proof that it compactly generates the DG category of $D$-modules on $V$ supported on $X$, for any closed embedding of $X$ into a smooth affine variety $V$ (one of the ways to define $D\Modu_X^{dg}$).

We recall the seminormalisation $X^{sn}$ and explain that: \begin{theorem}
There is a derived equivalence between the categories
$$D\Modu_X^{dg}\cong D\Modu_{X^{sn}}^{dg}.$$ 
\end{theorem}

Note that $X$ is cuspidal if and only if $X^{sn}$ is smooth. Moreover, if $X$ is Cohen--Macaulay, the equivalence is the derived version of the one explained in \cite{Cusps_and_D-modules} (and our argument in general is essentially the same as theirs).

In Section \ref{Calculation_of_cohomology_in_the_hypersurface_case}, in the case of a hypersurface, from Theorem \ref{introtheorem1} we derive the interesting identity: 
\begin{corollary}
If $X=\{f=0\}$ is a smooth or cuspidal hypersurface, then $\frac{D_{\mathbb{A}^n}}{D_{\mathbb{A}^n}\cdot f+f \cdot D_{\mathbb{A}^n}}=0$.
\end{corollary} 

This at first glance can be surprising. The formula computes $\Ext^1(D_X,D_X)$, and it is vanishing if and only if $\Diff(X)=\Diff(X)^{dg}$, that is, if $D$-modules on $X$ are the same as $\Diff(X)$-modules. For general $f$, we compute $H^\bullet(\Diff(X)^{dg})=\Ext^\bullet(D_X,D_X)$ and its action on $D$-modules (more precisely, on $\Ext^\bullet(D_X,M))$. \\

In Section \ref{curve}, we give some examples of our formulas and theorems in the case of regular holonomic $D$-modules when the variety is a curve $C$. By calculating Ext$^1(D_C,M)$, we show:
\begin{theorem}
The abelian subcategory of regular holonomic $D$-modules with `\textit{completely non-trivial monodromy}' around each non-cuspidal singularity over a curve maps to ordinary modules over $\Diff(C)$, \textit{i.e.}, they have no higher cohomology.
\end{theorem} The converse also holds for simple $D$-modules on seminormal curves. Here, \textit{`completely nontrivial monodromy'} means that, in the normalisation of $C$, all eigenvalues of monodromies about exceptional points are not equal to 1 (see below).\\

Finally in Section \ref{quotient}, we study the case of holonomic $D$-modules on isolated quotient singularities. In this case the results actually have the opposite flavour: 
\begin{theorem}
For an isolated finite quotient singularity $X$, local systems correspond to ordinary $\Diff(X)$-modules if there is trivial monodromy about singularities. The converse holds for simple $D$-modules, or more generally, intersection cohomology $D$-modules.\\
\end{theorem}

In \cite{Hotta1984TheIH} and later in \cite{Poisson_Traces_and_D-Modules_on_Poisson_Varieties} (see also \cite{Poisson_traces_D-modules_and_symplectic_resolutions}) a certain quotient $M(X)$ of our compact generator $D_X$ was considered which governs the invariants under Hamiltonian flows. It was used to define a new homology theory which fuses Poisson homology with the de Rham cohomology, which is particularly nice in the case of symplectic singularities. Other quotients of $D_X$ were studied in \cite{etingof2012coinvariants}, relating to other geometric structures on $X$. We hope that our study will have applications to these quotients and plan to address this elsewhere.\\

\textbf{Acknowledgement.} We are grateful to J. Bernstein for his original motivating question of whether $D_X$ is a compact generator, which triggered this work. We thank D. Gaitsgory for directing us to Theorem \ref{Gaits}. We also thank D. Ben-Zvi for explaining the various ways of defining the derived category of $D$-modules and N. Arbesfeld, A. Bode, D. Kaplan, Y. Lekili, T. Stafford and A. Yekutieli for useful discussions and suggestions. This paper is written with the full support from the author's PhD supervisor T. Schedler. 
The author was supported by an Industrial Strategy EPSRC scholarship at Imperial College London.\\

\begin{convention}
By a variety $X$, we always mean a reduced separated scheme of finite type over $\mathbb{C}$. We will always work with ({right}) quasi-coherent $D$-modules. 

We denote the straight $D_X$ for the canonical $D$-module (see below). When $X$ is smooth, we denote the curly $\mathcal{D}_X$ for the (sheaf) of rings of differential operators on $X$. When $X$ is affine, we reserve the notation Diff($X$) for the ring of Grothendieck differential operators on $X$.

We use the term local system to mean an $\mathcal{O}$-coherent (right) $D$-module (equivalently, a vector bundle with a flat connection) on a locally closed smooth subvariety. We use the term topological local system to mean a representation of the fundamental group of such a subvariety. By the Riemann-Hilbert correspondence, the latter is equivalent to the former when we require that the connection have regular singularities. We write $\IC(X)$ for the intermediate extension of the trivial local system.
\end{convention}

In most scenarios, $X$ will also be affine and it will have an embedding into $\mathbb{A}^n$.

\section{$D$-modules as modules over a DG algebra}\label{$D$-modules as modules over a DG algebra}
Although when $X$ is singular the category of ${D}$-modules on $X$ can no longer be viewed as the category of modules over $\Diff(X)$, in this section, we find a nice substitute. This substitute will in general be a DG algebra rather than a usual ring. 

In \cite{Cusps_and_D-modules}, they showed in the case when $X$ has only cuspidal singularities, one can still use the ring of differential operators $\Diff(X)$ and the abelian category of $D$-modules on $X$ is equivalent to $\Diff(X)$-mod. We show that our DG algebra will reduce to $\Diff(X)$ and our equivalence reduces to the derived version of theirs in this case.

This section is divided into three parts: the first part deals with the general case, the second part deals with the cuspidal case and the last part deals with a vanishing result that we will need for Section \ref{quotient}. 

\begin{subsection}{General case}

Suppose $X$ is affine. We can choose $i:X\hookrightarrow V$ a closed embedding into a smooth affine variety $V$ (most of the time $V=\mathbb{A}^n$); note that if $X$ is smooth, we can just take $V$ to be $X$. {We define the Kashiwara's category $D\Modu_X$ of $D$-modules on $X$ to be the full subcategory of $D$-modules on $V$. It can be shown that this definition does not depend on the embedding $i$ (\cite[Corollary A.9]{Poisson_traces_D-modules_and_symplectic_resolutions}). } We define the following $D$-module on $X$: $D_X=\mathcal{I}_X {D}_V\backslash D_V$, where $\mathcal{I}_X$ is the defining ideal of $X$. 
This is clearly a right ${D}_V$-module that is supported on $X$, hence by Kashiwara's definition, a $D$-module on $X$. 
The module $D_X$ has the defining property Hom$(D_X,M) = \Gamma_X(M)$, the vector space of sections of $M$ scheme-theoretically supported on $X$ {(\emph{i.e.}, annihilated by $\mathcal{I}_X$ for some $n$)}. 
Note that if $X$ and $V$ are smooth, $D_X$ is just the usual transfer module ${D}_{X\to V}$. 
{The object $D_X$ does not depend on the choice of embedding.
Indeed, given two closed embeddings $i_k:V\hookrightarrow V_k$ for $k=1,2$, let ${\mathcal{I}_k}_X$ is the ideal defining $X$ in $V_k$ and $D_{X,k}:={\mathcal{I}_i}_X {D}_{V_i}\backslash D_{V_i}$. One can check that the equivalence of categories in \cite[Theorem A.8]{Poisson_traces_D-modules_and_symplectic_resolutions} will send $D_{X,1}$ to $D_{X,2}$.}
When $X$ is not affine, we glue the categories of the open subsets $U_i$ of a covering together to obtain a canonical abelian category of $D$-modules on $X$. 
The local objects $D_{U_i}$ glue together in a canonical way to get a global $D$-module. See \cite[Section A.2]{Poisson_traces_D-modules_and_symplectic_resolutions} and \cite[Section 1.7.2]{Introduction_to_algebraic_D-modules}.

We recall that an object $E$ in the derived category $\mathcal{T}$ of an abelian category $\mathcal{A}$ is called a generator if Hom$(E[i],C)=0$ for all $i\in \mathbb{Z}$, implies $C=0$. The category $\mathcal{T}$ is called cocomplete if it has arbitrary direct sums. An object $C \in \mathcal{T}$ is called
compact if Hom$(C,-)$ commutes with direct sums. See \cite[Section 2.1]{Categorical}.

We fix $X$ with a closed embedding into $V$. Let $\mathcal{A}=D_V\text{-mod}_X$ be the abelian category of quasi-coherent $D$-modules on $V$ supported on $X$ and $D^b_\mathcal{A}(D_V)$ be the full subcategory of $D^b(D_V)$ consisting of complexes with cohomology sheaves supported on $X$. 

We recall the following theorem from Gaitsgory--Rozenblyum:

\begin{theorem}\cite[Proposition 4.7.3]{D-modules_and_crystals}\label{Gaits}

Suppose $X$ is a variety, with a closed embedding into $V$. Then the inclusion functor 

\begin{equation}\label{Gaitsgory_Theorem}
    i:D^b(\mathcal{A})\to D^b_\mathcal{A}(D_V)\tag{$\dagger$}
  \end{equation}
is an equivalence of categories. In particular, it is fully faithful.
\end{theorem}

Note that in \cite{D-modules_and_crystals}, it is stated that $i:D^b(D^b_\mathcal{A}(D_V)^{\heart})\to D^b_\mathcal{A}(D_V)$ is an equivalence, but this is clearly an equivalent statement to the one above.

This result can be thought of as an analogue of Beilinson's result for perverse sheaves: the derived category of the abelian category of perverse sheaves is the derived constructible category.

This theorem is important because it shows that two natural derived categories of $D$-modules are equivalent.\\

The theorem below is a special case of a result of Gaitsgory--Rozenblyum, as we will explain, but with a more explicit proof.

\begin{theorem}\label{ourtheo}
Let $X$ be an affine variety, then the module $D_X$ is a compact generator in $D^b_\mathcal{A}(D_V)$.
\end{theorem}

\begin{proof}
Generation: To show it is a generator, it is enough to observe that $D_X$ has a nonzero map to every $D$-module $M$ supported on $X$, as then there is a map from $D_X$ to a bounded complex starting $M$. Take $M$ to be a non-zero $D$-module supported on $X$, then because every element is annihilated by $\mathcal{I}_X^n$ for some $n$, for $0\neq m\in M$, we can choose $n$ to be such that $\mathcal{I}_X^n \cdot m=0$ and $\mathcal{I}_X^{n-1} \cdot m\neq 0$, choose $m'\in \mathcal{I}_X^{n-1} \cdot m$. Hence there is a non-zero map sending $1\in D_X$ to this element $m'$.

Compactness: Recall compactness is equivalent to perfectness in derived categories. A perfect complex is a finite complex of locally projective objects. Since $\mathcal{O}_V$ has finite global dimension, we can take a finite projective resolution $P^\bullet$ of $\mathcal{O}_X$ as an $\mathcal{O}_V$-mod. Since $D_X=\mathcal{O}_X\otimes_{\mathcal{O}_V}D_V$, we have that $P^\bullet \otimes_{\mathcal{O}_V}D_V$ is a finite projective $D$-module resolution of $D_X$. This completes the proof. 
\end{proof}

\begin{remark}
If $Y$ is not affine, this construction will still produce a compact object which is locally a generator. 
\end{remark}

\begin{remark}
In \cite[Corollary 3.3.3]{D-modules_and_crystals}, the authors proved a more general statement than Theorem 2.2, for a general variety $X$ (not necessarily affine), replacing $\mathcal{O}_X$ by a compact generator $M$ of $\mathcal{O}_X$-mod, so that the compact generator of $D$-mod is the \textit{induction} of $M$.  This induction makes sense in general, but in the case that $X$ is embedded into a smooth affine variety $V$, it's $i_* M \otimes^\mathbb{L}_{\mathcal{O}_V} D_V$.
\end{remark}

An abelian category $\mathcal{A}$ is called a Grothendieck category if it has a g-object, small colimits and the filtered colimits are exact. Recall that an object $G \in \mathcal{A}$ is called a g-object if the functor $X\to \Hom_\mathcal{A}(G, X)$ is conservative, \textit{i.e.} $X \to Y$ is an isomorphism as soon as $\Hom(G, X) \to \Hom(G, Y)$ is an isomorphism. In the case of a cocomplete abelian category, this is equivalent to saying that every object $X$ of $\mathcal {C}$ admits an epimorphism  $G^{(S)}\rightarrow X$, where $G^{{(S)}}$ denotes a direct sum of copies of $G$, one for each element of the (possibly infinite) set $S$. Such an object $G$ is usually called a generator, but we already used this term previously. For more detail, see \cite[Section 2.4]{Categorical}.

Note that the abelian category of quasi-coherent $D$-modules on an affine variety $X$ is a Grothendieck category.  This fact mentioned in \cite[Section 4.7]{D-modules_and_crystals}, but we give more detail here.
We only need to show if it has a g-object, as the other axioms are obvious. We let $G=\bigoplus_n I^n D_V\backslash D_V$, where $i:X\to V=\mathbb{A}^n$ is a closed embedding and $I$ is the defining ideal. This is a g-object since if $M$ is supported on $X$, then every element is killed by some element in $I^n$. It implies that there is a surjective map from $G^{(S)}$ to $M$.

\begin{remark}
As any Grothendieck category has enough injectives, the above implies that the abelian category of $D$-modules on a variety $X$ has enough injectives. 
\end{remark}

We recall the following fact about Grothendieck categories (see \cite{derivingdgcat}, but we are using the version found in \cite[Proposition 2.6]{Categorical}): 

\begin{proposition}\label{ourprop}
Let $\mathcal{A}$ be a Grothendieck category such that the triangulated category
$D(\mathcal{A})$ has a compact generator $E$. Denote by $A$ the DG algebra REnd$(E)$. Then the
functor RHom$(E,-)$ : $D(\mathcal{A}) \to D(A$-mod$)$ is an equivalence of categories. 
\end{proposition}

Here $D(A$-mod$)$ denotes the derived category of right {DG} modules. 

\begin{remark}
The inverse functor is given by $M\mapsto M\otimes_{\REnd{(E)}}E$.
\end{remark}

Notice that in the case $H^{\bullet} (A)$ is bounded in degree, the functor RHom$(E,-)$ sends $D^b(\mathcal{A}) \to D^b(A$-mod$)$. This will also be an equivalence since the inverse also sends $D^b(A$-mod$)\to D^b(\mathcal{A})$.  

Combining Theorem \ref{Gaits}, Theorem \ref{ourtheo} and Proposition \ref{ourprop}, we get the following corollary.

\begin{corollary}
There is an equivalence of categories between the bounded derived category of quasi-coherent $D$-modules on an affine variety $X$ and $D^b(A$-mod$)$, where $A$ is the DG algebra $\REnd(D_X)$.
\end{corollary}

This gives a proof of Theorem \ref{introtheorem1}: define $\Diff(X)^{dg}:=\REnd(D_X)$ and $D\Modu_X^{dg}$ as one of the categories in the equivalence (\ref{Gaitsgory_Theorem}) of Theorem \ref{Gaits}. And we get a triangulated equivalence $$D\Modu_X^{dg}\cong \Diff(X)^{dg}\Modu.$$

\begin{remark}
If $X$ is not affine, but $X\hookrightarrow V$ is globally embedded, we will still have an equivalence of $D(\mathcal{A})\to \mathcal{RE}nd(D_X)$-mod, a category of sheaves of modules on $X$.
\end{remark}

The following description of the (underived) endomorphisms of $D_X$, proved by an explicit computation on $V$, has been known to experts for a long time (see \cite[Theorem 15.3.15]{Noncommutative_Noetherian_Rings}, \cite[Theorem 1.7.1]{Introduction_to_algebraic_D-modules}). While the formula is old, we got the idea to think of it in terms of the object $D_X$ from \cite{Introduction_to_algebraic_D-modules}.

\begin{theorem}\label{eti}
There is a canonical filtered isomorphism  $\phi:  \text{End}(D_X)\to \text{Diff}(X)$.
\end{theorem}

Therefore, the higher DG structure of REnd($D_X$) `detects' singularities and it serves as a `correction' and `error term' to $\Diff(X)$. 

\end{subsection}

\begin{subsection}{Cuspidal Case}

We now turn the attention to the cuspidal case.  We say $f:Y\to X$ is a \textit{cuspidal quotient morphism} if it is a universal homeomorphism and $X$ and $Y$ are Cohen--Macaulay. It is a \textit{good} cuspidal quotient morphism if, in addition, a certain local cohomology sheaf vanishes, which will be automatically satisfied if $X$ (or $Y$) is a smooth variety, see \cite[Section 2]{Cusps_and_D-modules}. We say a Cohen--Macaulay variety $X$ is cuspidal if there is a cuspidal quotient morphism from a smooth variety to $X$. In the curve case, this is equivalent to the normalisation map is a bijective resolution of singularities. Examples include the normalization map of a curve with cusp singularities, the normalization map $\mathfrak{h}\to X_m$ of the space of
quasiinvariants for a Coxeter group, and the Frobenius homeomorphism in characteristic p, see \cite[Section 1.2]{Cusps_and_D-modules}. There are also examples from geometry of Lie algebras, see \cite[Theorem 4.4]{Deformations_of_symplectic_singularities_and_Orbit_method_for_semisimple_Lie_algebras}.

We can calculate $f^!D_X$ in this case. Consider $Y\to X \hookrightarrow \mathbb{A}^n$, and both $Y$ and $\mathbb{A}^n$ are smooth. {Let $I$ be the ideal defining $X$ and denote $\Diff_X(M,N)$ to be $\mathcal{O}_X$-linear differential operators from $M$ to $N$, where $M$ and $N$ are $\mathcal{O}_X$-modules.} 
Then 
\begin{align*}
    f^!D_X=& I{D}_{\mathbb{A}^n}\backslash{D}_{\mathbb{A}^n}\otimes_{{D}_{\mathbb{A}^n}}{D}_{\mathbb{A}^n\leftarrow Y}\\
    =&I{D}_{\mathbb{A}^n\leftarrow Y}\backslash{D}_{\mathbb{A}^n\leftarrow Y}\\
    =&\Diff_{\mathbb{A}^n}(\mathcal{O}_{Y},I)\backslash\Diff_{\mathbb{A}^n}(\mathcal{O}_{Y},\mathcal{O}_{\mathbb{A}^n})\\
    =&\Diff_{\mathbb{A}^n}(\mathcal{O}_{Y},\mathcal{O}_X)\\
    =&\Diff_{X}(\mathcal{O}_{Y},\mathcal{O}_X).
\end{align*} 

In \cite{Cusps_and_D-modules}, they used $\Diff(X)$-modules rather than Kashiwara's category $D\Modu_X$. They showed that the two approaches are equivalent in the good cuspidal case in Corollary 4.4 of their paper. Their key results in the affine case are summarised in the following theorem:

\begin{theorem}\label{bzn}

If $f:Y\to X$ is a good cuspidal quotient morphism, then the followings hold:
\begin{enumerate}[label={(\arabic*)}]
    \item $D\Modu_Y\cong D\Modu_X$ via $f^!$ and $f_*$.
    \item $\Diff(Y)\Modu\cong \Diff(X)\Modu$ induced by tensoring with  transfer bimodules ${D}_{X\leftarrow Y}$, ${D}_{X\to Y}$.
    \item If $Y$ is smooth, $D\Modu_X\cong \Diff(X)\Modu$.
\end{enumerate}
\end{theorem}

Note that (3) follows from (1) and (2), and we will strengthen (1) in Proposition \ref{prop}.

Since, by our definition, varieties are reduced, in the curve case the CM condition is automatically satisfied. The theorem generalises the curve case result found in \cite{DIFFERENTIAL_OPERATORS_ON_AN_AFFINE_CURVE} saying that
the category of $D$-modules on a cuspidal curve is Morita equivalent to the category of $D$-modules on its (smooth) normalization. This is a generalisation, because for cuspidal curves, the normalisation map is a universal homeomorphism, which is a cuspidal quotient morphism.

The transfer bimodules ${D}_{X\leftarrow Y}$, ${D}_{Y\to X}$ of the equivalence between $\Diff(X)\Modu$ and $\Diff(Y)\Modu$  coincide with the usual transfer bimodules when $Y$ and $X$ are smooth. More generally, Corollary 2.14 of \cite{Cusps_and_D-modules} says ${D}_{X\leftarrow Y}=\Diff_{X}(\mathcal{O}_{Y},\mathcal{O}_X)$. This is shown to be projective as a left module over $\Diff(X)$ and as a right module over $\Diff(Y)$ in the Morita equivalence in Theorem 4.3 of \cite{Cusps_and_D-modules}. 

Because $f$ induces an equivalence {between $D\Modu_Y$ and $D\Modu_X$}, we must have
\begin{align*}
    \Ext^i_{D\Modu_X}(D_X,M)&=\Ext^i_{D\Modu_Y}(f^!{D}_X,f^!{M})\\
    &=\Ext^i_{D\Modu_Y}(\Diff_{X}(\mathcal{O}_{Y},\mathcal{O}_X),f^!{M})\\
    &=\Ext^i_{\Diff(Y)\Modu}({D}_{X\leftarrow Y},\mathcal{O}_X),f^!{M})\\
    &=0,
\end{align*}
{where for the last line, we used that ${D}_{X\leftarrow Y}$ is projective in the good cuspidal case.} {And hence in the good cuspidal case, the functor $M\mapsto \RHom(D_X,M)$ from $D^b(\mathcal{A})$ to $D^b(\REnd(D_X)\Modu)$ is actually an abelian functor, \emph{i.e.}, it restricts to a functor of abelian categories from $D\Modu_X$ to End$(D_X)$-mod (which is $\Diff(X)$-mod by Theorem \ref{eti}). And our $D_X$ is mapped to $\Diff(X)$. }

\begin{remark}
It is shown in Remark 4.5 of \cite{Cusps_and_D-modules} that $f^!M = M \otimes_{{D}_X} {D}_{X\leftarrow Y}$ in the good cuspidal case. We can use this to calculate the corresponding $\Diff(X)$-module of our module $D_X$ under the equivalence of \cite{Cusps_and_D-modules} in this case. Let $i:X\to \mathbb{A}^n$ be the inclusion, the corresponding $\Diff(X)$ module is 
\begin{align*}
    i^!D_X&=I_X{D}_{\mathbb{A}^n}\backslash D_{\mathbb{A}^n}\otimes_{{D}_{\mathbb{A}^n}}{D}_{\mathbb{A}^n\hookleftarrow X}\\
    &=I_X{D}_{\mathbb{A}^n\hookleftarrow X}\backslash D_{\mathbb{A}^n\hookleftarrow X}\\
    &=\Diff_{\mathbb{A}^n}(\mathcal{O}_X,I_X)\backslash \Diff_{\mathbb{A}^n}(\mathcal{O}_X,\mathcal{O}_{\mathbb{A}^n})\\
    &=\Diff_{\mathbb{A}^n}(\mathcal{O}_X,\mathcal{O}_{X})\\
    &=\Diff_{X}(\mathcal{O}_X,\mathcal{O}_{X})\\
    &=\Diff(X)
\end{align*}

Since they both send the compact projective generator $D_X$ to the same compact projective generator $\Diff(X)$, our equivalence will reduce to the equivalence in \cite{Cusps_and_D-modules} in the cuspidal case.
\end{remark}

In summary, we have the following proposition:

\begin{proposition}
If $X$ has only cuspidal singularities, then $\REnd(D_X)\cong \End(D_X)\cong \Diff(X)$; furthermore, the functor $\REnd(D_X,-)$ in Proposition \ref{ourprop} is the derived functor of the abelian equivalence in Theorem \ref{bzn}.
\end{proposition}

The $D$-module equivalence in Theorem \ref{bzn} (1) can be easily generalised to remove the good condition at the cost of getting a derived equivalence rather than an abelian equivalence. We are going to use this proposition in Section \ref{curve}.

\begin{proposition}\label{prop}
Suppose $f:Y\to X$ is a universal homeomorphism, then there is an (derived) equivalence between $D$-modules on $Y$ and $D$-modules on $X$.
\end{proposition}
We will mimic the proof of \cite[Proposistion 3.14]{Cusps_and_D-modules}.
\begin{proof}
Consider the Cartesian diagram 

\begin{center}
\begin{tikzcd}
Y\times_X Y \arrow[r, "p_1"] \arrow[d, "p_2"] & Y \arrow[d, "f"] \\
Y \arrow[r, "f"]                              & X.               
\end{tikzcd} 
\end{center}
Since $f$ is proper (because $f$ is a universal homeomorphism, it is universally closed and of finite type, and we have assumed separatedness already), we can use proper base change: ${p_1}_* {p_2}^!={f}^!{f}_*$, see \cite[Proposition 5.4.2]{Gaitsgory_2013}. {Note that a $D\Modu$ on $Y\times_X Y$ is the same as a $D\Modu$ on $Y\times Y$ set-theoretically supported on the diagonal $\Delta:Y \to Y\times Y$. Indeed $(Y\times_X Y)^{\text{red}}=\Delta(Y)^{\text{red}}$. And the maps $p_2^!$ and $p_{1*}$ can be identified with the pullback and pushforward of $\pi:\Delta(Y)\to Y$. Therefore, by Kashiwara's equivalence, we have that ${p_1}_* {p_2}^!=\Id$ and hence by proper base change ${f}^!{f}_*=\Id$.}

To show that ${f}_*{f}^!=\Id$, note that we have a natural map $f_*f^!N\to N$, complete the cone, we have $f_*f^!N\to N\to M$, and thus we get $f^!f_*f^!N\to f^!N\to f^!M$, which is $f^!N\to f^!N\to f^!M$ by the above paragraph. Hence $f^!M=0$. Note that $f$ is surjective and dominant. Suppose that $M \neq 0$. Let $Z$ be an irreducible component of the (reduced) support of $M$. By passing to a smooth dense subset, we may assume that $Z$ is smooth. Let $z$ be the generic point of $Z$. Then $M_z$ is a nonzero vector space over the field $\mathcal{O}_{Z,z}$. By dominance, there is an irreducible component $Z'$ of {$Y$} such that for its generic point $z'$, the induced map $\mathcal{O}_{Z,z} \to \mathcal{O}_{Z',z'}$ is an injection (a field extension).  Therefore, $f^! M_{z}$ is given as:

$$O_{Z',z'} \otimes_{\mathcal{O}_{Z,z}} f^{-1}(M_{z}) [\dim X - \dim Y],$$
which is nonzero (and concentrated in degree $\dim X - \dim Y$), which is a contradiction.
\end{proof}

\begin{remark}
It follows from Theorem 4.3 and Remark 4.5 of \cite{Cusps_and_D-modules} that in the good cuspidal case, this is in fact an abelian equivalence. 
\end{remark}

We now recall \textit{seminormalisation}. If $X$ is a variety, then the seminormalisation $X^{sn}$ is the initial object in the category of universal homeomorphisms $Y\to X$. Note that by definition a curve is cuspidal if and only if its seminormalisation coincides with its normalisation. 

\begin{remark}
A variety $X$ is cuspidal if and only if $X^{sn}$ is smooth. Indeed, by definition $X^{sn}$ is smooth implies that $X$ is cuspidal. Conversely, if $Y \to X$ is a universal homeomorphism from a smooth variety $Y$, then by the universal property, there is a map $X^{sn} \to Y$, and since this is a finite birational map (as it factors through one), as $Y$ is normal (as it is smooth), Zariski's main theorem implies that $X^{sn}\cong Y$ and hence $X^{sn}$ is smooth.  
\end{remark}

We have the following corollary.

\begin{corollary}{\label{corollary2}}
There is an (derived) equivalence between $D$-modules on $X$ and $D$-modules on $X^{sn}$.
\end{corollary}

\subsection{Vanishing Ext for $D$-modules}




For a general $X$, we have the following vanishing result:

\begin{proposition}\label{vanishingprop}
If $M$ is supported at a point then $\ESxt^i(D_X,M)=0$ for $i\geq 1$.
\end{proposition}


\begin{proof}
If $M$ is supported at a point, then $M$ is just a direct sum of delta modules. {We can restrict to a local calculation in a formal neighbourhood $V$ of the origin and assume without loss of generality that the point is the origin and that} $M=\mathbb{C}[\partial_{x_1},\dots,\partial_{x_{n}}]$. As it is shown in \cite{matsumura1987}, it is the injective hull of $\mathbb{C}$ in $\mathbb{C}[\![{x_1},\dots,{x_n}]\!]$ and hence injective as a $\mathbb{C}[{x_1},\dots,{x_n}]$-module. Then by adjunction, $\Ext^i_{D_V}(D_X,M)\cong \Ext^i_{\mathcal{O}_V}(\mathcal{O}_X,M)$, {which is 0 as $M$ is injective.}\qedhere
\end{proof}

\begin{remark}
More generally, the injective dimension of $M$ as an $\mathcal{O}$-module is at most $\dim \supp M$ for a general $D$-module. See \cite{Injective_dimension_of_D-modules_a_characteristic-fre_approach} and \cite{Lyubeznik1993}. Therefore we get $\Ext^i(D_X, M) = 0$ for $i > \supp(M)$.
\end{remark}

\end{subsection}

\section{Calculation of cohomology in the hypersurface case}\label{Calculation_of_cohomology_in_the_hypersurface_case}

Now, we restrict to the case where we have a hypersurface $X$ that is a cut out by a single equation $f$ in $\mathbb{A}^n$. We wish to calculate the cohomology of RHom$(D_X,D_X)$ or more generally RHom$(D_X,M)$, where $M$ is a $D$-module supported on $X$. We will see that Ext$^i(D_X,M)$ will vanish for $i\geq 2$. And in the hypersurface case we can write down the formula for Ext$^1(D_X,M)$ easily once we have the correct derived category of $D$-modules supported on $X$. 

There is a free resolution $0\to D_{\mathbb{A}^n}\to D_{\mathbb{A}^n}\to D_X\to 0$, where the first map is applying multiplication by $f$ on the left and the second map is the quotient map. 

Then clearly we can replace the object $D_X$ with its free resolution $D_{\mathbb{A}^n}\xrightarrow{f\cdot} D_{\mathbb{A}^n}$. We have that  $\RHom_{D^b(\mathcal{A})}(D_X,M)$ is isomorphic to $$\RHom_{D^b_\mathcal{A}(D_V)}(D_{\mathbb{A}^n}\xrightarrow{f\cdot} D_{\mathbb{A}^n},M).$$ Since $D^b_\mathcal{A}(D_V)$ is defined as the full subcategory,
this is 
$$\RHom_{D^b(D_V)}(D_{\mathbb{A}^n}\xrightarrow{f\cdot} D_{\mathbb{A}^n},M).$$ The complex is then $M\to M$ where the map now is applying multiplication by $f$ on the right. Therefore, Ext$^0(D_X,M)=(M)^f$ and Ext$^1(D_X,M)=M/M\cdot f$. In particular Ext$^1(D_X,D_X)=\frac{D_{\mathbb{A}^n}}{D_{\mathbb{A}^n}\cdot f+f \cdot D_{\mathbb{A}^n}}$. Note this is only a vector space, no longer a $D_{\mathbb{A}^n}$-module. So, in particular we see that if $X$ is smooth, then $\frac{D_{\mathbb{A}^n}}{D_{\mathbb{A}^n}\cdot f+f \cdot D_{\mathbb{A}^n}}=0$, which is perhaps not obvious from direct computation. Note that Ext$^{\geq 2}(D_X,D_X)=0$.

Recall we have shown that Ext$^1(D_X,M)=0$ for cuspidal $X$. Then perhaps even more surprisingly, this shows for cuspidal singularities, $\frac{D_{\mathbb{A}^n}}{D_{\mathbb{A}^n}\cdot f+f \cdot D_{\mathbb{A}^n}}$ is also $0$, for which we can't find a pure algebraic proof. Furthermore, this shows our DG algebra does not `detect' all the singularities.

To summarise, we have the following:
\begin{formula} \label{formula}
In the case of $X$ is defined by a single equation $f$, we have that
\begin{enumerate}
    \item Hom$(D_X,M)=(M)^f$.
    \item Ext$^1(D_X,M)=M/Mf$.
    \item Ext$^{\geq 2}(D_X,M)=0$.
    \item When $X$ is either smooth or a cuspidal hypersurface (automatically CM) then $\frac{D_{\mathbb{A}^n}}{D_{\mathbb{A}^n}\cdot f+f \cdot D_{\mathbb{A}^n}}=0$.
\end{enumerate}
\end{formula}

Even though Ext$^1(D_X,D_X)$ is in general not a $\mathcal{D}_{\mathbb{A}^n}$-module, it is however still an RHom$(D_X,D_X)$ module, where now we have viewed it as a module over a DG algebra. In general Ext$^1(D_X,M)$ will be a module over Ext$^\bullet (D_X,D_X)$. An interesting question is when Ext$^1(D_X,M)$ is zero. In this case, RHom$(D_X,M)\cong \Hom (D_X,M)$ can be viewed as an ordinary module over $\Diff(X)$ (with possible higher $A_\infty$ structure). We will see later that some regular holonomic modules will have this property. 
\\

We can calculate the action explicitly:
\begin{align*}
   \text{REnd}(D_X)=&\text{RHom}(D_{\mathbb{A}^n}\to D_{\mathbb{A}^n},D_{\mathbb{A}^n}\to D_{\mathbb{A}^n})\\
   =&D_{\mathbb{A}^n}\to D_{\mathbb{A}^n}\oplus D_{\mathbb{A}^n}\to D_{\mathbb{A}^n},
\end{align*}
where the first arrow and second arrow are given explicitly by 
\begin{equation}{\label{1}}
    a\mapsto fa\oplus af,\alpha\oplus\beta\mapsto-f\beta+\alpha f
\end{equation}

Note this is quasi-isomorphic to 
\begin{center}

\begin{tikzcd}
D_{\mathbb{A}^n}\arrow[r]\arrow[d]& D_{\mathbb{A}^n}\oplus D_{\mathbb{A}^n}\arrow[r]\arrow[d]& D_{\mathbb{A}^n}\arrow[d]\\
0\arrow[r]& D_X\oplus 0\arrow[r,"-\cdot f"]&D_X,
\end{tikzcd}
\end{center}
where the vertical maps are either the zero map or the projection map. 

Multiplication is given by $$(a_1,b_1\oplus c_1,d_1)(a_2,b_2\oplus c_2,d_2)=(a_1c_2+b_1a_2,a_1d_2+b_1b_2\oplus c_1c_2+d_1a_2,c_1d_2+d_1b_2)$$ 
and its action on RHom$(D_X,M)=M\to M$ is given by 

\begin{equation}{\label{2}}
    (e,g)\cdot (a,b\oplus c,d)=(ga+eb,gc+ed)
\end{equation}

Now we can calculate the action of End$(D_X)=(D_X)^f$ on Hom$(D_X,M)=M^f$ and Ext$^1(D_X,M)=M/Mf$, via the action and the quasi-isomorphism described above. 


From (\ref{1}), we get the expression End$(D_X)=\frac{\{(\alpha,\beta)|\alpha f=f\beta\}}{\{f\cdot a\oplus a\cdot f| a\in D_{\mathbb{A}^n}\}}$.

End$(D_X)$ is a quotient of a subspace of $D_{\mathbb{A}^n}\oplus D_{\mathbb{A}^n}$ and Ext$^0(D_X,M)$ is a subspace of $M$, therefore the action is induced by (\ref{2}): $e\cdot (b\oplus c)=eb$. Note the isomorphism $\frac{(\alpha,\beta)|\alpha f=f\beta}{\{f\cdot a\oplus a\cdot f| a\in D_{\mathbb{A}^n}\}} \to (D_X)^f$ is given by projecting to the first coordinate and sending $\beta$ to zero. The inverse isomorphism is given by first lifting to an element $\alpha$ in $D_\mathbb{A}^n$ then solving the equation $\alpha f=f\beta$ for $\beta$ (the solution is unique since $f$ has no torsion), then sent to the quotient. The condition that multiplying on right by $f$ gives zero ensures that such a $\beta$ must exist. Since the action only depends on the first coordinate, we see that the action of End$(D_X)$ on Ext$^0(D_X,M)$ is the usual multiplication. 
 
On the other hand, the action on Ext$^1(D_X,M)$ is a bit more complicated; it has a \textit{twist}. Ext$^1(D_X,M)$ is a quotient of $M$ and the action is induced by $g\cdot(b\oplus c)=gc$. Therefore, the action is given by the $\beta$ after solving the equation in the previous paragraph. It is worthwhile to remark that $\beta$ has the same \textit{principal symbol} as $\alpha$, so after taking the associated graded module (with respect to either the additive or the geometric grading), the action is the same as without taking the twist. 

The action of Ext$^1(D_X,D_X)$ on Ext$^1(D_X,M)$ is induced by $e\cdot d=ed$. Upon identifying Ext$^1(D_X,D_X)$ with $D_X/D_Xf$ we see that the action is just the usual multiplication.

\begin{example}
We can compute $\Ext^1(D_X,D_X)$ explicitly when $X=\Spec\mathbb{C}[x,y]/(xy)$, the union of two axes in the plane. An explicit basis is given by $$P(\partial_x, \partial_y), y \partial_y P(\partial_x, \partial_y),$$ where $P$ is a monomial in $\partial_x, \partial_y$. We will use the Diamond Lemma, see \cite[Appendix A]{Schedler_2016}. Choose any ordering such that $x>y$. Note that Ext$^1(D_X,D_X)=D_{\mathbb{A}^2}/xyD_{\mathbb{A}^2}+D_{\mathbb{A}^2}xy$. 

The relations are the span of: 
\begin{enumerate}
    \item $xyg$, where $g$ is a monomial of $x,y,\partial_x,\partial_y$ (from $xyD_X$); 
    \item $xg$, where $\partial_x$ is not a factor of $g$;  
    \item $yg$, where $\partial_y$ is not a factor of $g$;
    \item and $(b x \partial_x + a y \partial_y) g + ab g$, for all $g$,  with  $a = 1+\ddeg_{\partial_x} g $ and $b = 1+\ddeg_{\partial_y} g$ (from $D_{\mathbb{A}^2}xy$ then subtract from $xyD_{\mathbb{A}^2}$).
\end{enumerate}

Using the last line allows us to get rid of all multiples of $x \partial_x$; in what remains, we can get rid of all multiples of $x$ by the second line.

But there are some redundancies: if we have $(xy \partial_x) g$, we get this to zero by the first line, or to $-(a+1)y((1/b)y \partial_y + 1)g$ by the last line.

This means that we can also get rid of multiples of $y^2 \partial_y$.  We can get rid of all $y$'s when not a multiple of $\partial_y$.

These are all the redundancies in applying the reductions $$xyg \to 0,$$ $$xg \to 0,$$ when $g$ is not a multiple of $\partial_x, yg \to 0$ when $g$ is not a multiple of $\partial_y$, and $$x\partial_x g \to -a((1/b)y \partial_y + 1)g.$$ So by the Diamond Lemma, we have found a basis consisting of the remaining expressions, which are of the form $P(\partial_x, \partial_y), y \partial_y P(\partial_x, \partial_y)$.

From this basis we see that if we look at the dimensions of the filtered pieces with respect to the additive filtration, the sequence of dimensions is a sum of shifted triangular numbers with one in deg 0, one in deg 2. Namely, the sequence will be 1,3,7,13,21,31...
\\

We can also compute the action of End($D_X$) on Ext$^1(D_X,D_X)$. Note that elements in End($D_X$) are $h\in D_X$ such that $hxy=0$ \textit{i.e.} $hxy=xyg$ for some $g\in D_{\mathbb{A}^2}$. Also note that $$x^iy^j\partial_x^n\partial_y^mxy=x^iy^j(xy\partial_x\partial_y+by\partial_y+mx\partial_x+nm)\partial_x^{n-1}\partial_y^{m-1}.$$ So $h$ has representatives 
\begin{enumerate}
    \item $x^iy^j\partial_x^n\partial_y^m$, where $i,j\geq 1$
    \item $y^j\partial_y^m$, where $j\geq 1$, and
    \item $y^i\partial_y^n$, where $i\geq 1$, 
\end{enumerate}
in $D_{\mathbb{A}^2}$. And as we noted before, the action is given by multiplication of the corresponding $x^{i-1}y^{j-1}(xy\partial_x\partial_y+ny\partial_y+mx\partial_x+nm)\partial_x^{n-1}\partial_y^{m-1}$ with the elements of Ext$^1(D_X,D_X)$. Note that the action is completely determined by the action of $x \partial_x^n$ and $y \partial_y^m$ since they generate End($D_X$). The corresponding elements are $x\partial_x^n+n\partial_x^{n-1}$ and $y\partial_y^m+m\partial_y^{m-1}$. So the action is $$x \partial_x^n\cdot \partial_x^i\partial_y^j=x\partial_x^{n+i}\partial_y^j+n\partial_x^{n-1+i}\partial_y^j=n\partial_x^{n-1+i}\partial_y^j$$
$$y \partial_y^m\cdot \partial_x^i\partial_y^j=y\partial_x^{i}\partial_y^{j+m}+m\partial_x^{i}\partial_y^{m-1+j}$$
$$x \partial_x^n\cdot y\partial_y\partial_x^i\partial_y^j=xy\partial_x^{n+i}\partial_y^{j+1}+ny\partial_x^{n-1+i}\partial_y^{j+1}=ny\partial_x^{n-1+i}\partial_y^{j+1}$$
$$y \partial_y^m\cdot y\partial_y\partial_x^i\partial_y^j=y^2\partial_x^{i}\partial_y^{j+m+1}+my\partial_x^{i}\partial_y^{m-2+j}.$$
\end{example}

Note that this example explicitly shows that Ext$^1(D_X,M)$ does not vanish in general. 

\section{Holonomic $D$-modules on curves}\label{curve}

In this section, we calculate $\ESxt^1(D_X,M)$ for $M$ a (regular) holonomic module on a curve $X$. In the general case, we show that if $M$ is simple and has nontrivial monodromies in the normalisations of the preimages of the non-cuspidal singularities then $\ESxt^1(D_X,M)$ vanishes, and we conjecture the converse direction is also true for simple $M$. We prove the conjecture in the case of $X$ is a planar multicross singularity. We also explain why all curves have category of $D$-modules derived equivalent to a curve with planar multicross singularities.

We can calculate the stalk of $\ESxt^1(D_X,M)$ in formal neighbourhoods. But the formal neighbourhood can be computed in an analytic neighbourhood. If the point $p$ is a smooth or a cuspidal point, then by {Formula \ref{formula},} we have concluded that the $\ESxt^1(D_X,M)_p=0$. Therefore, the sheaf is concentrated at the non-cuspidal singular points. The singular points are isolated, so we see we only need to do the calculation locally around each non-cuspidal singular point, and $\ESxt^1(D_X,M)$ must be a direct sum of skyscraper sheaves. Thus assume without loss of generality that the curve is affine.

\subsection{General curve}
Let $j:U\to X$ is an open embedding with $U$ smooth (but possibly not affine).
Recall Corollary \ref{corollary2} about the equivalence between the categories of $D$-modules on $X$ and its seminormalisation $X^{sn}$.

\textbf{Observation}: Observe that $\RSHom(D_X, j_* N) = \RSHom(D_U, N) = \Gamma_U(N)$. In particular, $\ESxt^{> 0}(D_X, j_* N) = 0$ for all $D$-modules $N$ on $U$.

\begin{definition}
We call an intermediate extension \textit{clean} if the canonical morphism $\IC(N)\hookrightarrow H^0 j_*N$ is an isomorphism. 
\end{definition}

In particular if $M=\IC(N)=H^0j_*N$ is clean, then by the above observation, Ext$^1(D_X,M)=0$.

Note that for $N$ indecomposable, then $\IC(N)$ is {an} indecomposable module if and only if $H^0 j_! N$ (and $H^0 j_* N$) are isomorphic to $\IC(N)$. This is implied by cleanness, and equivalent to it when $j$ is affine. See \cite[Theorem 3.4.2, (1)]{D-modules_perverse_sheaves_and_representation_theory}.

An example of a clean extension from $\mathbb{G}_m$ to $\mathbb{A}^1$ is {$N=(d-\lambda/x)D_{\mathbb{G}_m}\backslash D_{\mathbb{G}_m}$ for $\lambda \not \in \mathbb{Z}$} because $j_*N$ is simple. An example of an unclean extension is $\IC(\Omega_{\mathbb{G}_m})=\Omega_{\mathbb{A}^1}$ because $j_*\Omega_{\mathbb{G}_m}=\mathbb{C}[t,t^{-1}]$.

For $n$-lines intersection in the plane this is more complicated. Recall that the pushforward functor for singular varieties is inherited from the pushforward functor for the ambient smooth varieties. Let $i\circ j$ be the inclusion map from $U$ to $\mathbb{A}^2$, where $j$ is the open embedding from $U=\mathbb{A}^1\backslash \{0\}\to\mathbb{A}^1$ and $i$ is the closed embedding from $\mathbb{A}^1\to \mathbb{A}^2$. So $(i\circ j)_* N=i_{\bullet} (j_* N\otimes_{D_{\mathbb{A}^1}} D_{{\mathbb{A}^1}\to{\mathbb{A}^2}})$, which is $j_*N\otimes_{\mathbb{C}}\mathbb{C}[\partial_y]$. Since $\delta_0$ is simple as a $D_{\mathbb{A}^1}$ module, if $j_*N$ is simple, then we will have that $j_*N\otimes_{\mathbb{C}}\mathbb{C}[\partial_y]$ is a simple $D_{\mathbb{A}^1}\otimes D_{\mathbb{A}^1}\cong D_{\mathbb{A}^2}$ module. Then the intermediate extension must be the pushforward. {But indeed when $\lambda \in \mathbb{Z}$,  $j_*N\cong\mathbb{C}[t,t^{-1}]$, and for $\lambda \not\in \mathbb{Z}$, $j_*N\cong(xd-\lambda)D_{\mathbb{A}^1}\backslash D_{\mathbb{A}^1}$. And the simple submodules are $\mathbb{C}[t]$ and $(xd-\lambda)D_{\mathbb{A}^1}\backslash D_{\mathbb{A}^1}$ respectively. We see that $(d-\lambda/x)D_{\mathbb{G}_m}\backslash D_{\mathbb{G}_m}$ is still clean but $\IC(\Omega_{\mathbb{G}_m})$ is not.}

Note that by the above discussion, cuspidal quotient morphisms preserve indecomposable objects as well as isomorphism classes of objects. So, the property of having a clean extension is preserved under cuspidal quotient morphisms.  Since we can detect having a clean extension locally, we see that a simple holonomic $D$-module has a clean extension if and only if, pulling back to the seminormalisation, all monodromies around all preimages of the singularities are nontrivial.  This in turn is equivalent to asking that, in the normalisation, all monodromies around all preimages of singularities are nontrivial.  

\begin{definition}
We say $M$ has \textit{completely non-trivial monodromy} if for every composition factor $L$ of $M$ which is not supported at a point, the monodromy of $L$ about the preimage under the normalisation map $\nu:\tilde{X}\to X$ of every non-cuspidal singularity of $X$ does not have 1 as an eigenvalue. 
\end{definition}

We have deduced the following: 

\begin{theorem}{\label{theo}}
Let $X$ be a curve with the normalisation $\nu: \tilde{X} \to X$ and $M$ be a regular holonomic $D$-module on $X$ with completely non-trivial monodromy. Then $\Ext^{\geq 1}(D_X, M) = 0$.
\end{theorem}


Therefore the abelian subcategory of regular holonomic $D$-modules with completely non-trivial monodromy over a curve maps to ordinary modules over $\Diff(X)$ (possibly with $A_\infty$ structure). These are not the only $D$-modules mapping to ordinary modules over $\Diff(X)$, e.g., $j_* \Omega$ maps to ordinary modules over $\Diff(X)$, but the latter doesn't live in the abelian subcategory, as $\Ker(j_*M\to \delta^n)=\IC(X)$ does not map to an ordinary $D$-module.

\begin{remark}
We also expect the converse of the previous theorem to be true for $M$ simple, \emph{i.e.}, $\Ext^{\geq 1}(D_X, M) = 0$ with $M$ holonomic implies that $M$ has completely non-trivial monodromy. However, if $f:W \to X$ is the map from $W=$ $n$-lines intersection at the origin to $X$, since the $f^! D_X$ is not $D_W$, the local calculation doesn't go through. The module $f^! D_X$ should look like ${D}_{X\leftarrow W}$, but since $X$ is not cuspidal, we cant apply the machinery from \cite{Cusps_and_D-modules}. For a general cuspidal quotient morphism $f: X\to Y$, the vanishing of Ext$^{>0}(D_X, f^! M)$ does not appear to imply Ext$^{>0}(D_Y, M)$ vanishing.  (However, as discussed above, if $X$ is smooth, or more generally $Y$ is cuspidal, then Ext$^{>0}(D_Y, M) = 0$ for all $D$-modules $M$ on $Y$.) Note that $f^! D_X$ is not projective because if it is then Ext$^1(D_X,M)=0$ for $X$ $n$-straight line intersecting at the origin and $M$ the trivial module. 
\end{remark}

\subsection{Planar multicross singularity case}

We take formal neighbourhood $U$ at the singular points {of a curve $X$}, and we take its normalisation map $\Tilde{U}\twoheadrightarrow U$. $\Tilde{U}$ is a normal 1 dimensional, therefore it is smooth. But a regular local complete 1 dimensional ring must be of the form $\mathbb{C}[[x]]$. So, each connected component of $\Tilde{U}$ must be of this form. The normalisation map will factor through the bijective map $Z:=\Spec\mathbb{C}[[x]]\times_\mathbb{C}\dots\times_\mathbb{C} \mathbb{C}[[x]]\to U$, where $\mathbb{C}[[x]]\times_\mathbb{C}\dots\times_\mathbb{C} \mathbb{C}[[x]]$ is the fibered coproduct and it is the coordinate ring of the $n$ axes in $\mathbb{A}^n$. This map is a universal homeomorphism and the source is the seminormalisation $U^{sn}$. Therefore, a seminormal curve is precisely
one whose singularities formally locally look like coordinate axes in an affine space. This is known as \textit{multicross singularity} in the literature, see \cite{leahy1981}. 

There is another map from the $n$ axes in $\mathbb{A}^n$ {(which we named $Z$)} to $n$ distinct (straight) lines {in the plane $\mathbb{A}^2$ (which we are going to call $W$)} (for example, via the matrix \[\begin{pmatrix} 1&1&1&\dots&1&1&0\\ 0&1&2&\dots&n-3&n-2&1
\end{pmatrix}.)\] By Proposition \ref{prop}, we get another derived equivalence of $D$-modules. So, we see that locally there is an equivalence between $D$-modules on $X$ and $D$-modules on {$W$}. Therefore, up to derived equivalence, we can understand the Ext-algebra locally (as demonstrated below).

\begin{remark}
We stress that these equivalences are not isomorphisms of (DG) rings.
\end{remark}

\begin{remark}\label{cal}
We can explicitly write down a concrete basis of Ext$^1(D_W,M)$ in the case of $\IC(\Omega_{\mathbb{G}_m})$ and $\IC((d-\lambda/x)D_{\mathbb{A}^1}\backslash D_{\mathbb{A}^1})$, when $W$ is $n$-lines intersecting at the origin. 
Recall that we had the formula Ext$^1(D_W,M)=M/Mf$, {where $f$ is the defining equation of $W$ in $\mathbb{A}^2$}.

For $N=\Omega_{\mathbb{G}_m}$, we see $\IC(N)=\mathbb{C}[x]\otimes_\mathbb{C} \mathbb{C}[\partial_y]$, viewed as a right module. The multiplication on the right is done component-wise, with the standard action on the $x$-part and $\partial_y^p\cdot y=p\partial_y^{p-1}$. So, applying multiplication by $x$ increases the exponent in the first coordinate and applying multiplication by $y$ decreases the exponent in the second coordinate. Say $f=y(x+\alpha_1 y)\cdots(x+\alpha_{n-1} y)$, $\alpha_i\in \mathbb{C}$ distinct. To calculate Ext$^1(D_X,M)$, we need to calculate $Mf$, the image. Clearly applying multiplication by $y$ is a surjection, and applying multiplication by $x+\alpha_iy$ we get $\IC(N)$ except a copy of $\mathbb{C}\otimes_\mathbb{C}\mathbb{C}[\partial_y]$. Continue like this, we see that Ext$^1(D_X,M)$ has a basis $\langle 1,x,\dots ,x^{n-2}\rangle\otimes_\mathbb{C}\mathbb{C}[\partial_y]$, which is non-zero if $n>1$ (\textit{i.e.} it is not cuspidal).

For $N=(d-\lambda/x)D_{\mathbb{G}_m}\backslash D_{\mathbb{G}_m}$, $\IC(N)=(xd-\lambda)D_{\mathbb{A}^1}\backslash D_{\mathbb{A}^1}\otimes_{\mathbb{C}}\mathbb{C}[\partial_y]$, which equals $\mathbb{C}[x^{\lambda}]$ as a vector space where $\mathbb{C}[x^{\lambda}]$ is a power series infinite in both directions. Therefore, applying multiplication by $x$ and $y$ is surjective. Therefore, $\Ext^1(D_X,M)= 0$.


\end{remark}

Recall that the support is a closed subset, and that {$M$ is simple implies the support $M$ is irreducible}. The irreducible closed subsets of $n$-lines intersecting at the origin are either a point or a copy of $\mathbb{A}^1$. Therefore, $\IC(\Omega_{\mathbb{G}_m})$, $\IC((d-\lambda/x)D_{\mathbb{G}_m}\backslash D_{\mathbb{G}_m})$ {and delta modules} are the only possible simple regular holonomic modules on $\mathbb{A}^1$. 

Therefore, we have the following:

\begin{proposition}
Let $X$ be a curve such that $X$ is formally locally equivalent to $n$-lines intersecting in a plane for each non-cuspidal point. If $M$ is a simple regular holonomic $D$-module on $X$, then Ext$^1(D_X,M)=0$ if and only if $M$ has non-trivial monodromy about each non-cuspidal singularity or $M$ is supported at a point.
\end{proposition}

\begin{proof}
As we discussed before, it is enough to consider formally locally. Then the calculation is done in the previous remark and Proposition \ref{vanishingprop}. 
\end{proof}

The examples include the nodal curve $\mathbb{C}[x(x-1),x^2
(x-1)] \subset \mathbb{C}[x]$; in this case $X=X^{sn}$, and the induced map from the intersection of 2 lines in a plane to $X$ is formally locally an isomorphism. In fact, every planar-multicross must be formally locally of this form. It does not include examples like formally locally non-transverse intersections. 

\begin{remark}
More generally, given a global curve $X$, one can always construct another curve $X^{pl}$, whose singularities are all formally locally lines intersecting in a plane, and whose category of $D$-modules is derived equivalent to that of $X$. Hence up to derived equivalence we can be in the situation of Proposition 4.3. Indeed, at each singular point $p_i\in X^{sn}$ and $U_i$ such that $p_i\in U_i$, we have $\mathcal{O}(\widehat{{(U_i)}_{p_i}})\cong \mathbb{C} [[x_{i,1}]]\oplus \dots\oplus \mathbb{C}[[x_{i,n_{i}}]]$. Then by Nakayama's lemma we pick $U_i$ small enough such that $\exists y_{i,1},\dots,y_{i,{n_i}}\subset\mathcal{O}(U_i)$ generate $\mathcal{O}(U_i)$ with $y_{i,j}\equiv x_{i.j}$ (mod $\mathfrak{m_i}^2$). This is true because the statement is true in $\mathcal{O}(U_i)_p$ and we can pick $U_i$ small enough that $y_{i,j}$ extend to $U_i$ and generate. Define $X^{pl}:=(X^{sn},\mathcal{O}^{pl})$, where $\mathcal{O}^{pl}\subset\mathcal{O}^{sn}$ is a sheaf of rings, defined by $f|_{U_i}\in\mathcal{O}^{pl}(U_i)$ if $f|_{U_i}\in \mathbb{C} \langle y_{i,1}+\dots+y_{i,n_{i}-1}, y_{i,2}+\dots+(n_i-2)y_{i,{n_{i}-1}}+y_{i,n_i}\rangle\subset \mathcal{O}(U_i)^{sn}$ for all $i$, the subalgebras generated by these two elements (we are using the matrix at the beginning of this section); along with the open set covering the smooth part, this satisfies the sheaf condition. We need to check that $U_i^{sn}\backslash \{p_i\}\cong U_i^{pl}\backslash \{p_i\}$, \textit{i.e.} $\mathcal{O}(U_i\backslash \{p_i\})$ is generated by $y_{i,1}+\dots+y_{i,n_{i}-1}, y_{i,2}+\dots+(n_i-2)y_{i,{n_{i}-1}}+y_{i,n_i}$ And there is a universal homeomorphism from $X^{sn}$ to $X^{pl}$. This is true because the singular locus of $U_i$ is just $p_i$ and $\forall q\in U_i\backslash \{p_i\}$, $\exists i$ such that $y'_{i,j}(q)\neq 0$, then by Nakajima's lemma every local neighbourhood in $U_i$ is generated by $y_{i,j}$, so $\mathcal{O}(U_i\backslash \{p_i\})$ is also generated by them. $X^{pl}$ is formally locally $n$ curves meeting transversely in a plane, Zariski locally some of these curves could coincide. Therefore there is a zigzag diagram
\begin{center}
    \begin{tikzcd}
  & X^{sn} \arrow[ld, two heads] \arrow[rd, two heads] &        \\
X &                                                    & X^{pl},
\end{tikzcd}\\
\end{center}

where each arrow is a universal homeomorphism and thus gives a derived equivalence of $D$-modules and the category is `nicer' to study in $X^{pl}$. 
\end{remark}

\section{Holonomic $D$-module on isolated quotient singularities}\label{quotient}

In this section, we calculate the image of holonomic $D$-modules of a smooth variety $V$ under our canonical map in Proposition \ref{ourprop} in the situation of isolated quotient singularities by a finite group $G$. {Let $X=V/G$ be the quotient of a variety $V$ by $G$,} we need to calculate $\Ext^{>0}(D_X,M)$. As usual, it suffices to compute in a formal neighbourhood of an isolated singularity. Therefore we can assume $V$ is the affine space $\mathbb{A}^n$. In particular, we will see that intermediate extension of the canonical $D$-module can be viewed as an ordinary module over the Grothendieck ring of differential operators $\Diff(V/G)$. And the intermediate extension of nontrivial local systems $L$ on $U = (V/G) \backslash\{0\}$ where $G$ acts freely away from the origin, have non-vanishing $\Ext^{\dim V-1}(D_X, IC(L))$. This is completely opposite to the curve case.

Let $\pi:V \to V/G$ be the quotient map. The strategy is the following: we first study the fundamental exact triangle associated to the intersection cohomology $D$-module and compute its cone just as we did in Section \ref{curve} (cf. Lemma \ref{5.2}). Using this, we can compute $\pi^!\IC(N)$ (cf. Proposition \ref{5.4}). Finally, using equivariant-adjunction, we can compute $\Ext^i(D_X,\IC(N))$ (cf. Theorem \ref{5.9}).
The results can be globalised, (cf. Remark \ref{globalremark}).\\

Recall that a pseudo-reflection is an invertible linear transformation $g:V\to V$ such that the order of $g$ is finite and the fixed subspace $ V^{g}$ has codimension 1. The Chevalley--Shephard--Todd theorem says:

\begin{theorem}\label{CST}
For $x\in V=\mathbb{A}^n$, let $G_x\subset G$ be the stabilizer of $x$. Then the quotient $V/G$ is smooth if and only if each $G_x$ is generated by pseudo-reflections.
\end{theorem}

Suppose that $K$ is the subgroup generated by pseudo-reflections (it is a normal subgroup). Then $V/K$ is smooth by the above theorem. Since $V/G=(V/K)/(G/K)$, where $V/K$ is smooth and {is isomorphic to $V$,} we can assume without loss of generality that $G$ contains no pseudo-reflections. We can also assume $\dim V\geq 2$ as otherwise we recover the curve case. 

For each point in the free locus $\Tilde{U}$, the stabilizer is trivial, therefore the condition of Theorem \ref{CST} is satisfied. Thus the quotient of the free locus is smooth. And the singular locus is exactly the image of the complement of the free locus and it has codimension at least 2 by assumption. Let $U$ be the smooth locus, the image of $\Tilde{U}$. As $G$ contains no pseudo-reflections, {locally this is a smooth covering map, }we see that $\pi^!D_U=D_{\tilde{U}}$ and $\pi^!\mathcal{O}_U=\mathcal{O}_{\tilde{U}}$.

We begin the calculation of the Ext groups. This requires a few steps:

\begin{lemma}\label{5.2}
{Let $X=V/G$ and $j:U\hookrightarrow X$ be the inclusion of the open smooth locus.} The cone $K$ in the exact triangle $\IC(X)\to j_*\Omega_U\to K\to$ is isomorphic to $\delta[1-  d]$, where $d$ is the dimension of $X$ and $\delta$ is the delta module supported at the singularity.
\end{lemma}

\begin{proof}
Recall that for $i>0$, $$\Ext^i(\IC(X),\delta)=\Ext^i(\delta,\IC(X))=H^{d-i}_{dR}(U)$$ by \cite[Lemma 4.3]{Poisson_Traces_and_D-Modules_on_Poisson_Varieties}. Since $G$ is finite and we are calculating cohomology with coefficients in $\mathbb{C}$, applying the Cartan-Leray spectral sequence to $\Tilde{U}\simeq S^{2d-1}$ (homotopic) and $G$ we see that $H^{d-i}_{dR}(U)\cong H^{d-i}_{dR}(S^{2d-1})$. 

The exact triangle $\IC(X)\to j_*\Omega_U\to K$ induces the exact sequence $$\Ext^i(\delta,j_*\Omega_U)\to \Ext^i(\delta,K)\to\Ext^{i+1}(\delta,\IC(X))\to\Ext^{i+1}(\delta,j_*\Omega_U),$$ and $$\Ext^i(\delta,j_*\Omega_U)\cong\Ext^i(j^!\delta,\Omega_U)=0$$ as $j^!\delta=0$. Therefore $$\Ext^i(\delta,K)\cong \Ext^{i+1}(\delta,\IC(X))\cong H^{d-i-1}_{dR}(S^{2d-1}).$$ Since $K$ is concentrated at the singularity, it is a direct sum of $\delta$'s, so $\Ext^i(\delta, K)$ gives the multiplicity of $\delta[-i]$ in $K$ (for $i \geq 0$). Therefore $K=\delta[1-d]$.
\end{proof}

\begin{remark}
It follows similarly that for holonomic $D$-modules $\IC(N)$, we have $\Ext^m(\IC(N),\delta)\cong H^{d-m}_{dR}(U,N)^*$, and the cone $K$ in $\IC(N)\to  j_*N\to K$ is isomorphic to $\delta\otimes_{\mathbb{C}} H^{d-1-*}_{dR}(U,N)$. When $N$ is a non-trivial topological local system $L$, we can compute this further. $\Ext^m(\IC(N),\delta)$ is then $0$. Indeed, $H^i(U,L)\cong H^i(S^{2d-1},\pi^*L)$, hence $H^{0<i<2d-1}(U,L)=0$ and {if} $L$ is non-trivial then $H^0(U,L)=0$; finally by an Euler characteristic consideration, $H^{2d-1}(U,L)=0$. This proves that for an isolated singularity and a non-trivial simple topological local system $L$, $j_*L=\IC(L)$, \textit{i.e.}, it is a clean extension.
\end{remark}

\begin{proposition}\label{propo}\label{5.4}
$\pi^!\IC(X)=\Omega_V$ and $\pi^!\IC(L)=j_*\Omega_{\Tilde{U}}$ for $L$ a non-trivial simple topological local system.
\end{proposition}

\begin{proof}
Note that we have {an} exact triangle $\pi^!\IC(X)\to\pi^!j_*\Omega_U\to \pi^!K\to$. Base change formula implies $\pi^!j_*\Omega_U\cong j_*\Omega_{\Tilde{U}}$. Let $\Tilde{K}$ be the cone of $\IC(V)\to j_*\Omega_{\Tilde{U}}\to \Tilde{K}\to$. The proof of the previous Lemma implies that $\Tilde{K}\cong \delta[1-d]$, where $\delta$ is supported at the preimage of the isolated singularity. Using the base change formula we see that $\pi^!K\cong \Tilde{K}$. In other words, we have $$\pi^!\IC(X)\to j_* \Omega_{\Tilde{U}}\to \Tilde{K}\to.$$ We also have $\Omega_V=\IC(V)$, therefore there is another exact triangle $$\Omega_V\to j_*\Omega_{\Tilde{U}}\to \Tilde{K}\to.$$ 

Recall $\Tilde{K}=\delta[1-d]$. We can compute $\Hom(j_*\Omega_{\Tilde{U}},\Tilde{K})$. The triangle $\Omega_V\to j_*\Omega_{\Tilde{U}}\to \Tilde{K}$ induces $$\Ext^{i-1+d}(\delta,\delta)\to \Ext^i(j_*\Omega_{\Tilde{U}},\delta)\to\Ext^i(\Omega_V,\delta)\to \Ext^{i+d}(\delta,\delta).$$ Note $\Ext^i(\delta,\delta)=\mathbb{C}$ if $i=0$ and 0 otherwise, and $$\Ext^i(\IC(V),\delta)\cong\Ext^i(\delta,\IC(V))\cong H^{d-i}_{dR}(\Tilde{U})\cong H^{d-i}_{dR}(S^{2d-1})$$ for $i>0$. So $\Ext^i(j_*\Omega_{\Tilde{U}},\delta)=\mathbb{C}$ if $i=d$ or $i=-d+1$, \textit{i.e.}, $\RHom(j_*\Omega_{\Tilde{U}},K)\cong \mathbb{C}\oplus\mathbb{C}[2d-1]$. In particular $\Hom(j_*\Omega_{\Tilde{U}},\Tilde{K})=\mathbb{C}$. 
We need to show that the two arrows of $j_*\Omega_{\Tilde{U}}\to\Tilde{K}$ are non-zero, 
it will follow that the two triangles must be the same up to rescaling, hence $\pi^!\IC(X)\cong \Omega_V$.

The arrow of $j_*\Omega_{\Tilde{U}}\to\Tilde{K}$ of the second triangle is non-zero as otherwise it will imply that $\Omega_V\cong j_*\Omega_{\Tilde{U}}\oplus\delta[-d]$ which is impossible. 

For the first triangle, note that $\IC(X)$ is indecomposable implies that the map $j_* \Omega_U \to K$ is non-zero. 
As the map we want to is non-zero is the image of $j_* \Omega_U \to K$ under $\pi^!:\Hom(j_* \Omega_U, K) \to \Hom(\pi^! j_* \Omega_U, \pi^! K)$, 
it suffices to show that $\pi^!$ is an injection between the Hom sets.
By adjunction this is identified with
$\Hom(j_* \Omega_U, K) \to \Hom(\pi_! \pi^! j_* \Omega_U, K)$, composing with the natural map
$\pi_! \pi^! j_* \Omega_U \to j_* \Omega_U$.
But $$\pi_! \pi^! j_* \Omega_U = \pi_! j_* \Omega_{\tilde {U}} = \pi_* j_* \Omega_{\tilde {U}} = j_* \pi_* \Omega_{\tilde {U}} = j_* (\bigoplus L),$$
summing over all rank one local systems $L$ on $U$.
Under this identification, $\pi_! \pi^! j_* \Omega_U \to j_* \Omega_U$ becomes the projection map $j_* (\bigoplus L)\to j_* \Omega_U$, which is the only non-zero map up to scaling.
Since this is a projection map, we obtain that $\pi^!$ is indeed an injection between the Hom sets.

For the second statement, we use the fact that it is a clean extension and base change:
$$\pi^!\IC(L)=\pi^!j_*L=j_*\pi^!L=j_*\Omega_{\Tilde{U}}.$$
 
\end{proof}

\begin{remark}
Consider the commutative diagram
\begin{center}
    \begin{tikzcd}
\tilde{U} \arrow[r, "j"] \arrow[d, "\pi"] & V \arrow[d, "\pi"] \\
U \arrow[r, "j"]                          & X,                 
\end{tikzcd}
\end{center}
one sees that $\IC(\pi_*\Omega_{\Tilde{U}})=\pi_*(\IC(V))$ (\cite[Example 4.43]{nikolaev}). Since on the smooth open $\tilde{U}$ we have $\IC(V)=j_*\Omega_{\Tilde{U}}$, it follows that $\IC(\pi_*\Omega_{\Tilde{U}})=\pi_*(\Omega_V)$. Moreover, as the free locus $\pi:\tilde{U}\to U$ is a covering, we know that $\pi_*\Omega_{\Tilde{U}}=\Omega_U\otimes_{\mathbb{C}}\mathbb{C}[G]$. Hence $\IC(X)=\pi_*(\Omega_V)^G$ and $\pi^!(\IC(X))=\pi^!(\pi_*(\Omega_V)^G)$. 
\end{remark}

\begin{lemma}
$\Diff(V/G)\cong\Diff(V)^G$
\end{lemma}

\begin{proof}
Every differential operator on $V/G$ gives a $G$-invariant differential operator on $V_{\text{free}}$, which by taking principal symbol gives a section of $\text{Sym}^mT_{V_{\text{free}}}$. As $G$ contains no pseudo-reflections, by Hartogs' theorem, these sections extend uniquely to $\text{Sym}^mT_{V}$, that is, principal symbols of differential operators in $\Diff(V)^G$. Therefore by induction on $m$, we can conclude the lemma. 
\end{proof}

The next lemma was first proved in \cite[Lemma 2.9]{Poisson_Traces_and_D-Modules_on_Poisson_Varieties} by considering the induction functor Ind from $\mathcal{O}$-modules to right $D$-modules on any algebraic variety and noting that $\pi_*\circ \text{Ind}=\text{Ind}\circ\pi_*$ for all proper morphisms $\pi$ (this follows from the two adjunctions Ind $\dashv$ Res and $\pi_*\dashv\pi^!$). Here we give a more direct proof. Note that neither proof uses the assumption that $G$ contains no pseudo-reflections.

\begin{lemma}\label{key}
$D_{X}=(\pi_*D_V)^G$.
\end{lemma}

\begin{proof}

Choose an embedding $i:X\to Y$, we write $\pi'$ for $i\circ\pi$.
We can compute 
\begin{align*}
    (\pi'_*D_V)^G & =\pi'_{\bullet} (D_{V\to Y})^G\\
    & =\pi'_\bullet (\mathcal{O}_V \otimes_{(\pi')^{-1}\mathcal{O}_Y} D_Y)^G\\
    & =\pi'_\bullet ({\mathcal{O}_V}^G \otimes_{(\pi')^{-1}\mathcal{O}_Y} D_Y)\\
    & = I_X\backslash D_Y\\
    & = i_*D_X,
\end{align*}
where $\pi_\bullet$ (respectively $\pi^{-1}$) is the sheaf-theoretic pushforward (respectively pullback).

\end{proof}

\begin{remark}
From the previous lemma we can conclude that $$\Hom((\pi_*D_V)^G,D_X)=\Hom(D_X,D_X)=\Diff(X),$$ which is $\Diff(V)^G=\Hom(D_V,D_V)^G$ when $G$ contains no pseudo-reflections. By adjunction, the left hand side also equals to $$\Hom((\pi_*D_V)^G,D_X)\cong\Hom(D_V,\pi^!D_X)^G.$$ Therefore we conclude that if $G$ contains no pseudo-reflections then $$\Gamma(V,\pi^!D_X)^G\cong\Gamma(V,D_V)^G.$$
\end{remark}

With the help of this lemma, we can now do the calculation. For a vector space $W$ with a group $G$ acting on it, denote $W^{G,\bot}:=\Ker(W^*\to (W^G)^*)=(W/W^G)^*$.

\begin{theorem}\label{5.9}
Let $N$ be a holonomic $D$-module on $U$, then $\Ext^{m}(D_X,\IC(N))\cong (\Gamma(V,\delta)\otimes(H^{-m-1}_{dR}(\tilde{U},\pi^!\mathbb{D}N)^{G,\bot}))^G$ for $m>0$. In particular, we have that $\Ext^{\bullet}(D_X,\IC(X))$ is concentrated in degree 0. 
\end{theorem}

\begin{proof}
Recall for any $D$-module $F$ on $V$, we have the exact triangle $$F\to \tilde{j}_*\tilde{j}^*F\to K\to,$$ where $K=\tilde{i}_*\tilde{i}^!F[1]$. Taking the long exact sequence, we have that $$\to\Ext^p(D_V,F)\to R^p\Gamma(\tilde{U},F)\to \Gamma(\delta)\otimes H^{p}(\tilde{i}^*F)\to.$$ Therefore we have an isomorphism $R^p\Gamma(\tilde{U},F)\cong \Gamma(\delta)\otimes H^{p}(\tilde{i}^*F)$ for $p>0$. 

For $M$ a $D$-module on $X$, we have the following diagram 

\begin{center}
\begin{tikzcd}
\pi^!M \arrow{r}              & \tilde{j}_*\tilde{j}^*\pi^!M \arrow{r}           & \tilde{i}_*\tilde{i}^!\pi^!M[1] \arrow{r}           & {} \\
H^0\pi^!M \arrow{r} \arrow{u} & \tilde{j}_*\tilde{j}^*\pi^!M \arrow{r} \isoarrow{u} & \tilde{i}_*\tilde{i}^!H^0\pi^!M[1] \arrow{r} \arrow{u} & {}.
\end{tikzcd}
\end{center}

As $\Ext^{>0}(D_V, \pi^!M)=R^{>0}\Gamma(\text{cone}(H^0\pi^!M\to \pi^!M))$, from the diagram it follows that $\Ext^{>0}(D_V, \pi^!M)=R^{>0}\Gamma(\text{cone}(\tilde{i}_*\tilde{i}^!H^0\pi^!M\to\tilde{i}_*\tilde{i}^!\pi^!M))$.

In the case when $M=\IC(N)$, where $N$ is a holonomic $D$-module on $U$, we know that $H^{m}i^!M\cong H^{-m}_{dR}(U,\mathbb{D}N)^*$ for $m\geq 0$ from the fact that $H^{-m}i^*M\cong H^{-m}_{dR} (U,N)$. Therefore, as $i^!\delta=\mathbb{C}$ and either $\IC(\pi^!N)=H^0\pi^!M$ or $H^0\pi^!M$ is an extension of $\IC(\pi^!N)$ by $\delta$, we see that $H^{m}\tilde{i}^!H^0\pi^!M\cong H^{-m}_{dR}(\tilde{U},\pi^!\mathbb{D}N)^*$ for $m> 0$. Using $H^{-m}_{dR}({U},\mathbb{D}N)^*=H^{-m}_{dR}{(\tilde{U},\pi^!\mathbb{D}N)^G}^*$, we see that for $m>0$,
\begin{align*}
    \Ext^{m}(D_X,\IC(N))\cong&\Gamma H^m\text{cone}(\tilde{i}_*\tilde{i}^!H^0\pi^!M\to\tilde{i}_*\tilde{i}^!\pi^!M)^G\\
    \cong&\Gamma H^m\text{cone}(\tilde{i}_*\tilde{i}^!\IC(\pi^!N)\to\tilde{i}_*\pi^!\tilde{i}^!\IC(N))^G\\
    \cong&(\Ker(H^{-m}_{dR}(\tilde{U},\pi^!\mathbb{D}N)^*\to H^{-m}_{dR}{(\tilde{U},\pi^!\mathbb{D}N)^G}^*)[1]\otimes\Gamma(V,\delta))^G\\
    \cong& (\Gamma(V,\delta)\otimes(H^{-m-1}_{dR}(\tilde{U},\pi^!\mathbb{D}N)^*/H^{-m-1}_{dR}({U},\mathbb{D}N)^*))^G\\
    \cong& (\Gamma(V,\delta)\otimes H^{-m-1}_{dR}(\tilde{U},\pi^!\mathbb{D}N)^{G,\bot})^G.
\end{align*}

Note that $H^{m}_{dR}(\tilde{U},\pi^!N)^*=H^{m}_{dR}({U},N)^*$ if and only if $G$ acts trivially on $H^{m}_{dR}(\tilde{U},\pi^!N)$. Therefore we see that for the trivial local system, $\Ext^{>0}(D_X,\IC(X))=0$. 

\end{proof}

\begin{corollary}
In the special case of $N=E$ for $E$ a vector bundle on $U$ with a flat algebraic connection, we see that 

$$\Ext^{m}(D_X,\IC(E))\cong (\Gamma(V,\delta)\otimes(H^{d-m-1}_{dR}(\tilde{U},\pi^!\mathbb{D}E)^{G,\bot}))^G$$ for $m>0$.
\end{corollary}

Since $\tilde U$ is homotopic to $S^{2d-1}$, and $\pi^! E$ is a trivial rank-one local system on which $G$ acts by a character, we get that $H_{dR}^{<d}(\tilde U, \pi^! \mathbb{D}E)$ is concentrated in degree zero and gives there the inverse of the character by which $G$ acts on $L$. This shows that for the trivial local system, the higher Ext groups are always zero; for a non-trivial local system, the higher Ext groups are non-zero in the $d-1$ degree and equals to $\Gamma(V,\delta)_\chi$, where $\chi$ is the character by which $G$ is acting on $E$.

By Proposition \ref{propo}, we see that $\Ext^0(D_X,\IC(X))=(\Omega_V)^G$ for the trivial local system and $\Ext^0(D_X,\IC(L))=\Gamma(\tilde{U},\Omega_{\tilde{U}})_\chi$ (where $\chi$ is the defining representation for $L$) for a non-trivial simple topological local system $L$. 

Hence we see that the holonomic $D$-modules can be viewed as a $\Diff(V/G)$ module via our correspondence if at each singularity, either the local system has no monodromy, or it pulls back to a local system on $V$ which is nontrivial at the preimage of the singularity. This includes $\IC(X)$, but excludes $\IC(L)$ for non-trivial simple local systems.

This is the complete opposite to the case of curves in the Section \ref{curve}.

\begin{remark}\label{globalremark}
The above calculation can be globalised to the following. 
Suppose now $L$ is a local system on a locally closed subset $Z\subset X=V/G$ where $V$ now is a general variety such that $V/G$ has isolated singularities. 
Let $Z_x:=Z\cap B_x$, where $B_x$ is a small analytic ball around $x\in X$. 
Furthermore, let $\widetilde{Z_x}$ be a connected component of $\pi^{-1}(Z_x)$. 
Then the above calculation also shows that $$\ESxt^{m}(D_X,\IC(L))_x\cong H^{d-m-1}\Gamma(\widetilde{Z_x},\pi^!(\mathbb{D}L))^{G,\bot},$$ for $m>0$, which is 0 if and only if $H^{d-m-1}\Gamma(Z_x,\mathbb{D}L)\hookrightarrow H^{d-m-1}\Gamma(\widetilde{Z_x},\pi^!\mathbb{D}L)$ is surjective. 
If $L$ has regular singularity at $x$, this happens if and only if $L^{\pi_1(Z_x)}\cong L^{\pi_1(\widetilde{Z_x})}$, \textit{i.e.}, $G=\pi_1(Z_x)/\pi_1(\widetilde{Z_x})$ acts trivially on $L^{\pi_1(\widetilde{Z_x})}$. 

Therefore:
\begin{itemize}
    \item if $\dim(Z)\leq 1$, then $\Ext^{>0}(D_X,\IC(L))=0$;
    \item if $\dim(Z)\geq 2$, in general the higher Ext can be complicated. However, if the closure $\overline{\widetilde{Z_x}}$ is a local complete intersection, for each singular point $x$ and $X$ only consists regular singularities, then $\Ext^m(D_X,\IC(L))=0$ for $1\leq m \leq \dim(Z)-2$ as the link at an isolated complete intersection singularity is $(d-2)$-connected.
\end{itemize}

\end{remark}

\begin{remark}
The canonical $D$-module $M(X)$ of Etingof--Schedler \cite{Poisson_Traces_and_D-Modules_on_Poisson_Varieties} is a local enhancement of Poisson homology of $X$. When $X=V/G$ where $G$ is a finite subgroup of $Sp(V)$ and $X$ has isolated singularities, $M(X)$ is isomorphic to a direct sum of intermediate extensions of trivial local systems on each stratum, see \cite[Corollary 4.16]{Poisson_Traces_and_D-Modules_on_Poisson_Varieties}. We deduce that $M(X)$ can be seen an ordinary $D$-module over $\Diff(X)$. This includes the Du Val case.

\end{remark}

We can also do a similar calculation to deduce the structure of $\Ext^\bullet(D_X,D_X)$:

\begin{proposition}
The cohomology of the DG algebra $\REnd(D_X)$ is given by
$$\Ext^\bullet(D_X,D_X)\cong \Diff(V)^G\bigoplus (\langle(1-g)\cdot \Gamma(V,\delta)\rangle_{g\in G}\otimes \Gamma(V,\delta))^G[1-d].$$
\end{proposition}

\begin{proof}
Note we have the diagram
\begin{center}
\begin{tikzcd}
\pi^!D_X \arrow{r}              & \tilde{j}_*\tilde{j}^*\pi^!D_X \arrow{r}           & \tilde{i}_*\tilde{i}^!\pi^!D_X[1] \arrow{r}           & {} \\
H^0\pi^!D_X \arrow{r} \arrow{u} & \tilde{j}_*\tilde{j}^*\pi^!D_X \arrow{r} \isoarrow{u} & \tilde{i}_*\tilde{i}^!H^0\pi^!D_X[1] \arrow{r} \arrow{u} & {}.
\end{tikzcd}
\end{center}

Also, we have $H^0\pi^!D_X\cong D_V$. Indeed, as $D_V\cong H^0j_*j^*H^0\pi^!D_X$, we have by adjunction, a map $H^0\pi^!D_X\to D_V$. This map is injective as the kernel is concentrated at the origin and $\Hom(\delta,H^0\pi^!D_X)=\Hom(\delta,D_X)=\Hom(\delta,D_V)^G=0$. As $\Ext^1(\delta, M)=0$ and $D_V$ is indecomposable, we see this map has to be surjective too.

Furthermore, \begin{align*}
    i^!D_X\cong&\RHom(\mathbb{C},i^!D_X)\\
    \cong&\RHom(\delta,(\pi_*D_V)^G)\\
    \cong&\RHom(\delta,\pi_*D_V)^G\\
    \cong&\RHom(\pi^*\delta,D_V)^G\\
    \cong&(i^!D_V)^G\\
    \cong&\Gamma(V,\delta)^G[-d].
\end{align*}

Therefore, 
\begin{align*}
    \Ext^m(D_X,D_X)&\cong\Gamma H^m\text{cone}(H^0\pi^!D_X\to \pi^!D_X)^G\\
    &\cong \Gamma H^m \text{cone}(\tilde{i}_*\tilde{i}^!H^0\pi^!D_X\to \tilde{i}_*\tilde{i}^!\pi^!D_X)^G\\
    &\cong \Gamma H^m \text{cone}(\tilde{i}_*\tilde{i}^!D_V\to \tilde{i}_*\pi^!{i}^!D_X)^G\\
    & \cong H^{m-d}(\Ker(\Gamma(V,\delta)\to \Gamma(V,\delta)^G)[1]\otimes\Gamma(V,\delta))^G\\
    &\cong H^{m-d+1}(\langle(1-g)\cdot \Gamma(V,\delta)\rangle_{g\in G}\otimes\Gamma(V,\delta))^G,
\end{align*}
for $m>0$. Adding $\Ext^0(D_X,D_X)\cong\Diff(X)$ yields the result.
\end{proof}

\begin{remark}
Note that $\pi^!D_X\not\cong D_V$ (but $H^0\pi^!D_X\cong D_V$). If it is, then $$\Ext^\bullet(\pi^!D_X,\pi^!D_X)^G\cong \Ext^\bullet(\pi_*D_V,D_X)^G\cong \Ext^\bullet(D_X,D_X)$$ which is not concentrated in degree zero, but $\REnd(D_V)^G$ is.
\end{remark}

\begin{remark}
The proposition implies that any Kleinian singularity $X$ has non-vanishing $\Ext^1(D_X,D_X)$, hence they are not cuspidal.
\end{remark}

\begin{remark}
For $M=\IC(L_\chi)$ where $L_\chi$ is a simple non-trivial topological local system, there is an action $$(\langle(1-g)\cdot \Gamma(V,\delta)\rangle_{g\in G}\otimes\Gamma(V,\delta))^G\times \Gamma(V, \mathcal{O}_{\tilde{U}})^G\to (\Gamma(V,\delta)\otimes H^{-d}_{dR}(\tilde{U},\pi^!\mathbb{D}L_\chi)^{G,\bot})^G.$$
Note we have $H^{-d}_{dR}(\tilde{U},\pi^!\mathbb{D}L_\chi)^{G,\bot}=\chi$. On the subrepresentation $\chi\subset\Gamma(\delta)^{G,\bot}$, this action should be induced by applying distributions to functions $\chi\otimes\mathcal{O}_{\tilde{U}}\to \chi$, and the action should be zero elsewhere. 
\end{remark}


\begin{remark}
In the case when $X$ is a hypersurface, the argument simplifies and no longer needs Lemma \ref{key}. This case already includes the Du Val singularities.

Indeed, when $X$ is a hypersurface we have $\Ext^i(D_X,M)=0$ for all $i\geq 2$. So we are only interested in $\Ext^1(D_X,M)$. 
For the IC extension of the trivial monodromy, the exact sequence $\IC(X)\to j_*\Omega_U\to K$ induces $$0\to\Ext^1(D_X,\IC(X))\to \Ext^1(D_X,j_*\Omega_U)\to \Ext^0(D_X,\delta)\to 0.$$ To show $\Ext^1(D_X,\IC(X))=0$, we only need to show $\Ext^1(D_X,j_*\Omega_U)\to \Hom(D_X,\delta)$ is an isomorphism. We already know it is surjective, it remains to show it is injective. As of before,  $$\Ext^1(D_X,j_*(\Omega_U))=\Ext^1_D(D_U,\Omega_U)=\Ext^1_\mathcal{O}(\Omega_U,\Omega_U)=H^1_{Coh}(U,\Omega_U).$$ Consider a \v{C}ech covering of $U$ by open affines $V\backslash \{x_i-\text{axis}\}/G$. We get a sequence $$\to \mathbb{C}[x_1^{\pm 1},x_2,\dots,x_n]^{G}\oplus\dots\oplus\mathbb{C}[x_1,\dots,x_{n-1},x_n^{\pm 1}]^{G} \to \mathbb{C}[x_1^{\pm 1},\dots,x_n^{\pm 1}]^{G}\to 0,$$ the cokernel then can be identified with $(\frac{1}{x_1\dots x_n}\mathbb{C}[x_1^{-1},\dots,x_n^{-1}])^{G}$, these are exactly the elements of $\Hom(D_X,\delta)$. The map $\Ext^1(D_X,j_*\Omega_U)\to \Hom(D_X,\delta)$ is an $\End(D_X)\Modu$ map. Inside of $\End(D_X)\cong \Diff(X)$, there is a commutative subring generated purely by the partials (these two modules are actually rank 1 over this). A surjective endomorphism of finitely generated modules over a commutative ring is automatically an isomorphism.

For a non-trivial simple monodromy $L_\chi$, we have $$0\to\Ext^1(D_X,\IC(L_\chi))\to H^1_{Coh}(U,L_\chi)\to H^0_{dR}(U,L_\chi)\to 0.$$ But $H^0_{dR}(U,L_\chi)=0$, we have that $\Ext^1(D_X,\IC(L_\chi))\cong H^1_{Coh}(U,L_\chi)$ (where $L_\chi$ is the defining representation). Using the same covering but applied to $L_\chi$ similarly shows that $\Ext^1(D_X,\IC(L_\chi))=(\frac{1}{x_1\dots x_n}\mathbb{C}[x_1^{-1},\dots,x_n^{-1}])^\chi$.
\end{remark}

\begin{remark}
It might be possible to unify Section \ref{curve} and Section \ref{quotient}. In some sense, both sections are dealing with computing a cone of the form $H^0\pi^!M\to \pi^!M$ on $C$ for a finite map $\pi:C\to D$, then reducing to $D$. In Section \ref{curve}, the map was the normalisation map; In this section, the map is a quotient by a finite group $G$ and the reduction process is taking $G$-invariance. 
\end{remark}

\bibliographystyle{alpha}
{\footnotesize
\bibliography{ms}}
\end{document}